\documentclass[11pt]{amsart}
\usepackage{amsmath,amssymb,amsbsy,amsfonts,amsthm,latexsym,
                        amsopn,amstext,amsxtra,euscript,amscd,bm,stmaryrd}
\usepackage{url}
\usepackage[colorlinks,linkcolor=blue,anchorcolor=blue,citecolor=blue]{hyperref}
\usepackage{color}
\usepackage[english]{babel}
  \usepackage{dsfont}
\usepackage{geometry}
\begin{document}

\newtheorem{theorem}{Theorem}
\newtheorem{lemma}[theorem]{Lemma}
\newtheorem{algol}{Algorithm}
\newtheorem{cor}[theorem]{Corollary}
\newtheorem{prop}[theorem]{Proposition}

\newtheorem{proposition}[theorem]{Proposition}
\newtheorem{corollary}[theorem]{Corollary}
\newtheorem{conjecture}[theorem]{Conjecture}
\newtheorem*{definition*}{Definition}
\newtheorem{remark}[theorem]{Remark}

 \numberwithin{equation}{section}
  \numberwithin{theorem}{section}

\newtheorem*{theorem*}{Theorem}
\newtheorem*{question*}{Question}

\newcommand{\comm}[1]{\marginpar{%
\vskip-\baselineskip 
\raggedright\footnotesize
\itshape\hrule\smallskip#1\par\smallskip\hrule}}

\def\sssum{\mathop{\sum\!\sum\!\sum}}
\def\ssum{\mathop{\sum\ldots \sum}}
\def\iint{\mathop{\int\ldots \int}}
\newcommand{\twolinesum}[2]{\sum_{\substack{{\scriptstyle #1}\\
{\scriptstyle #2}}}}

\def\cA{{\mathcal A}}
\def\cB{{\mathcal B}}
\def\cC{{\mathcal C}}
\def\cD{{\mathcal D}}
\def\cE{{\mathcal E}}
\def\cF{{\mathcal F}}
\def\cG{{\mathcal G}}
\def\cH{{\mathcal H}}
\def\cI{{\mathcal I}}
\def\cJ{{\mathcal J}}
\def\cK{{\mathcal K}}
\def\cL{{\mathcal L}}
\def\cM{{\mathcal M}}
\def\cN{{\mathcal N}}
\def\cO{{\mathcal O}}
\def\cP{{\mathcal P}}
\def\cQ{{\mathcal Q}}
\def\cR{{\mathcal R}}
\def\cS{{\mathcal S}}
\def\cT{{\mathcal T}}
\def\cU{{\mathcal U}}
\def\cV{{\mathcal V}}
\def\cW{{\mathcal W}}
\def\cX{{\mathcal X}}
\def\cY{{\mathcal Y}}
\def\cZ{{\mathcal Z}}

\def\C{\mathbb{C}}
\def\F{\mathbb{F}}
\def\K{\mathbb{K}}
\def\Z{\mathbb{Z}}
\def\R{\mathbb{R}}
\def\Q{\mathbb{Q}}
\def\N{\mathbb{N}}
\def\M{\textsf{M}}

\def\({\left(}
\def\){\right)}
\def\[{\left[}
\def\]{\right]}
\def\<{\langle}
\def\>{\rangle}

\def\vec#1{\mathbf{#1}}

\def\e{e}

\def\eq{\e_q}
\def\fS{{\mathfrak S}}

\def\bfalpha{{\boldsymbol \alpha}}

\def\lcm{{\mathrm{lcm}}\,}

\def\fl#1{\left\lfloor#1\right\rfloor}
\def\rf#1{\left\lceil#1\right\rceil}
\def\mand{\qquad\mbox{and}\qquad}

\def\jt{\tilde\jmath}
\def\ellmax{\ell_{\rm max}}
\def\llog{\log\log}

\def\Qbar{\overline{\Q}}

\def\lam{\lambda}
\def\Yildirim{Y{\i}ld{\i}r{\i}m\ }
\def\Lam{\Lambda}
\def\lam{\lambda}

\title[On the minimal number of small elements generating prime fields]
{\bf On the minimal number of small elements generating prime fields}

\author{Marc Munsch}
\address{5010 Institut f\"{u}r  Analysis und Zahlentheorie
8010 Graz, Steyrergasse 30, Graz}
\email{munsch@math.tugraz.at}

\date{\today}


\begin{abstract}
 In this note, we give an upper bound for the number of elements from the interval $\[1,p^{1/4e^{1/2}+\epsilon}\]$ necessary to generate the finite field $\mathbb{F}_{p}$ with $p$ an odd prime. The general result depends on the localization of the divisors of $p-1$ and can be for instance used to deduce easily results on a set of primes of density $1$.

\end{abstract}

\bibliographystyle{srtnumbered}
\maketitle

\section{Introduction}

 E. Artin conjectured in 1927 that any positive integer $n>1$, which is not a perfect square, is a primitive root modulo $p$ for infinitely many primes  $p$. It remains open nowadays but was proved assuming the Generalized Riemann Hypothesis for some specific Dedekind zeta functions by Hooley in 
\cite{Hooley}. Using the development of large sieve theory leading to Bombieri-Vinogradov theorem, one can show that Artin's primitive root conjecture is true for almost all primes (see for instance \cite{HB} or \cite{Moreesurvey} for an extended survey about this conjecture). Another related classical problem is to bound the size $g(p)$ of the smallest primitive root modulo $p$. The best unconditional result is $g(p)=O(p^{1/4+\epsilon})$ and is due to Burgess in \cite{Burgess}, as a consequence of his famous character sum estimate. This is very far from what we expect, and assuming Generalized Riemann Hypothesis we can for instance show that $g(p)= O(\log^2 p)$ (see \cite{Montgomery} following \cite{Ankeny} or \cite{Tao} for results under the Elliott-Halberstam conjecture). Like before, as a consequence of the large sieve, the upper bound $g(p) = O ((\log p)^{2+\epsilon})$ is valid for almost all primes (see \cite{BurElliott}). The problem of improving unconditionnaly the bound on the least primitive root seems presently out of reach.  For instance, we cannot perform directly the Vinogradov trick 
to show that there exists a primitive root less than $p^{1/4\sqrt{e}+\epsilon}$, however we can reach that range for this following variant question:
 
 \begin{question*}How large should be $N$ (in terms of $p$) such that $<1,\dots,N>$ is a generating set of $\mathbb{F}_{p}^{*}$. \end{question*} Is is shown by Burthe in \cite{Burthegenerating} that $N=p^{1/4e^{1/2}+\epsilon}$ is sufficient\footnote{The result holds in fact for a composite $n$ as long as $8\nmid n$.} and seems to be the lower limit of what is possible unless Burgess character sum bound is improved. Nonetheless, in view of this result, several interesting related questions can be formulated. Harman and Shparlinski considered the problem of minimizing the value of $k$ such that for a sufficiently large prime $p$ and for any integer $a<p$, there is always a solution to the congruence $$ n_1 \dots n_k \equiv a \,\,(\bmod \,p),\,\,1 \leq n_1, \dots , n_k \leq N$$ and showed in \cite{HarShpar} that $k=14$ is an admissible value. \footnote{$k$=7 is admissible if we only ask that there is a solution for almost all values of $a$.} From an algorithmic point of view, another interesting question is to know precisely how many elements of $\{1,\dots,N\}$ are in fact necessary to generate the full multiplicative group. We consider in this note the problem of the size of a generating set consisting of small elements less than $N$. \begin{question*} How many elements of $\{1,\cdots,p^{1/4e^{1/2}+\epsilon}\}$ do we need in order to to generate $\mathbb{F}_{p}^{*}$?\end{question*}

$\vspace{1mm}$

Let $p$ be a prime such that $p-1=\# \mathbb{F}_{p}^{*}=q_1^{\alpha_1}\dots q_r^{\alpha_r}$ with $q_i,\, i=1,\dots, r$ the prime factors of $p-1$. We denote by $\omega(n)$ the number of prime factors of an integer $n$, therefore we have $\omega(p-1)=r$.

$\vspace{1mm}$
 
 The first elementary result in this direction is the following:

\begin{lemma}\label{elementary} For every $\epsilon>0$, we need only $\omega(p-1)$ elements among $\{1,\cdots,p^{1/4e^{1/2}+\epsilon}\}$ to generate $\mathbb{F}_{p}^{*}$. \end{lemma}

\begin{proof} Classically, using Burgess' character sums inequality (see \cite{Burgess}) combined with ``Vinogradov trick" (see \cite{Vino2}, \cite{Vino1}), we can pick up $x_1,\dots,x_r<p^{1/4e^{1/2}+\epsilon}$ such that $x_i$ is not a $q_i^{th}$ residue modulo $p$. Fixing $g$ a primitive root, we have $x_i=g^{\beta_i}$ with $\gcd(\beta_i,q_i)=1$. Thus, $\gcd(\beta_1,\dots,\beta_r)$ is coprime to $p-1$. By Bezout's theorem, there exists integers $l_1,\dots l_r$ such that $\displaystyle{\sum_{i=1}^{r} l_i\beta_i}$ is coprime to $p-1$. Hence, $x_1^{l_1}\dots x_r^{l_r}$ is a primitive root of $\mathbb{F}_{p}^{*}$ and the statement is proved.    \end{proof}

 In this note, we wonder if we could improve on this bound which means require less small elements to generate the full group. We will not be able to do this in full generality, the result depending on the anatomy of $p-1$. To measure this, we introduce the following definition.
 
 \begin{definition*} Let $l\geq 1$, we denote by $\omega_l(n)=\#\{q \text{ prime }, q\vert n, q\leq (\log (n+1))l^l\}$ the number of ``small" prime divisors of $n$.   \end{definition*}

  In the rest of the paper, $\log_k x$ will denote the $k$ times iterated logarithm when $k$ is an integer. Using a combinatorial argument and recent development in sieve theory in non regularly distributed sets, we prove in Section $3$:
 
 \begin{theorem*} Let $l:=l(p)\geq 1$ a parameter tending to infinity with $p$ such that $l\leq \frac{\log p}{1000 \log_2 p}$. We need $O\left(\omega_l(p-1) + \frac{\omega(p-1)}{\log l}\right)$ elements smaller than $p^{1/4\sqrt{e}+\epsilon}$ to generate the multiplicative group $\mathbb{F}_{p}^{*}$ where the implied constant is effective. \end{theorem*} We will also give a more precise result of this type and deduce, in the last part of the paper, stronger results for almost all primes. In the next section, we recall some sieve results that we will use in our argument.

\section{Sieve fundamental result}

In this section, we will use the notations and recall the setting of \cite{Matomakisieve}. Let $\mathbb{P}$ be the set of all primes and let $\mathcal{P} \subseteq \mathbb{P}$ be a subset of the primes $\leq x$. The most basic sieving problem is to estimate
$$
\Psi(x; \mathcal{P}) := \#\{ n \leq x \colon p \mid n \implies p \in \mathcal{P}\}.
$$
In other words we sieve the integers in $[1, x]$ by the primes in $\mathcal{P}^c = (\mathbb{P} \cap [1,x]) \setminus \mathcal{P}$. A simple inclusion-exclusion argument suggests that $\Psi(x; \mathcal{P})$ should be approximated by
$$
x \prod_{p \in \mathcal{P}^c} \left(1-\frac{1}{p}\right).
$$
This is always an upper bound, up to a constant, and a lower bound, up to a constant, if $\mathcal{P}$ contains all the primes larger than $x^{1/2-o(1)}$. 
 On the other hand, there are examples where $\Psi(x; \mathcal{P})$ is much smaller than the expected lower bound. For instance if one fixes $u \geq 1$ and lets $\mathcal{P}$ consist of all the primes up to $x^{1/u}$, then the prediction is about $x/u$ whereas, by an estimate for the number of smooth numbers, we know that $\Psi(x; \mathcal{P}) = \rho(u) x$ with $\rho(u) = u^{-u(1+o(1))} $ as $u \rightarrow \infty$, which is much smaller for large $u$.

The first ones to study what happens if one also sieves out some primes from $[x^{1/2}, x]$ were Granville, Koukoulopoulos and Matom\"aki~\cite{sieveworks}. They conjectured that the critical issue is to understand what is the largest $y$ such that
\begin{equation}
\label{eq:sum1/pydef}
\sum_{\substack{p \in \mathcal{P} \\ y \leq p \leq x^{1/u}}} \frac{1}{p} \geq \frac{1+\varepsilon}{u}.
\end{equation}
More precisely, they conjectured that when this inequality holds, the sieve works about as expected. On the other hand they gave examples with
$$
\sum_{y \leq p \leq x^{1/u}} \frac{1}{p} = \frac{1-\varepsilon}{u}
$$
such that $\Psi(x; \mathcal{P})$ is much smaller than expected.

We will use the following result proved by Matom\"aki and Shao confirming that conjecture:

\begin{theorem}
\label{th:MT}\cite[Theorem 1.1]{Matomakisieve}
Fix $\varepsilon>0$. If $x$ is large and $\mathcal{P}$ is a subset of the primes $\le x$ for which there are some $1 \leq u \leq v\le \frac{\log x}{1000 \log_2 x}$ with
$$
\sum_{\substack{ p\in \mathcal{P} \\ x^{1/v}<p\leq x^{1/u}}} \frac 1p \geq \frac{1+\varepsilon}{u},
$$
then
$$
\frac{\Psi(x;\mathcal{P})}{x} \geq A_v \prod_{p\in \mathcal{P}^c}  \left( 1 -\frac 1p \right),
$$
where $A_v$ is a constant with $A_v = v^{-v(1+o_\varepsilon(1))}$ as $v \to \infty$. If $u$ is fixed, one can take $A_v = v^{-e^{-1/u} v(1+o_\varepsilon(1))}$ as $v \to \infty$.
\end{theorem}

\section{Idea of the method and main results}

\begin{definition*}We define $h(p)$ as the number of elements smaller than $p^{1/4\sqrt{e}+\epsilon}$ which are sufficient to generate the multiplicative group $\mathbb{F}_{p}^{*}$.  \end{definition*}

 The aim of this note is to give improvements on the size of $h(p)$. The main idea is the following: due to the sparsity of powers, for large divisors $q_1,q_2$ of $p-1$, a non $q_1^{th}$ residue will be more likely a non $q_2^{th}$ residue. Thus, we do not need to pick up a non-residue for every power as it is done in Lemma \ref{elementary} and we can further play this game with more divisors in order to decrease the number of necessary steps in the argument. In order to do that, we will use the result on the sieve recalled in previous section. The dependance on $v$ in the lower bound of Theorem \ref{th:MT} will prevent us to regroup as much divisors as we want, thus we will carefully split the set of prime divisors in blocks of size $k$ with an ``optimal" value of $k$ coming from the application of Theorem \ref{th:MT}.

Given a parameter $l \geq 1$, we obtain a bound for $h(p)$ depending on $\omega_l(p-1)$. If for some relatively large $l$, $\omega_l(p-1)$ is small, this will give a significant improvement on the trivial bound $\omega(p-1)$ coming from Lemma \ref{elementary}.

The next result is the main tool that we are going to use to deduce to derive these improvements. It shows that we can handle several large prime divisors of $p-1$ simultaneously.

\begin{prop}\label{blocsk}\textnormal{[\bf{Main proposition}]}
Let $l:=l(p)\geq 1$ a parameter tending to infinity with $p$ and $k$ an integer verifying $k\leq \frac{\log l}{4}$. Moreover, assume that $l\leq \frac{\log p}{1000 \log_2 p}$. Suppose that $q_1,\dots,q_k$ are prime divisors of $p-1$ greater than $(\log p)l^l$. Then, for $p$ sufficiently large, there exists an integer $n\leq N=p^{1/4\sqrt{e}+\epsilon}$ which is a non $q_i^{th}$ residue for $i=1,\dots,k$.
\end{prop}

\begin{proof}
 Define $S=\{1\leq n\leq N \text{ s.t. }  n\text{ is a non } q_i^{th} \text{-residue modulo }p\text{ for } i=1,\dots,k\}$ and suppose that $S=\emptyset$ which means that every integer in this interval is $q_i^{th}$ residue modulo $p$ for at least one $i$. Thus, we have in particular \begin{equation}\label{decompprimes}\mathcal{P}=\{q \text{ prime}, 1\leq q\leq N\}= \bigcup_{i=1}^{k} \mathcal{P}_{i}\end{equation}  where $\mathcal{P}_{i}=\mathcal{P} \cap \{ q_i^{th} \text{-residue modulo }p \}  .$ 
For $x$ sufficiently large and $u,v$  parameters to be specified later, we have by Mertens' Theorem that, 

$$ \sum_{q\leq x } \frac{1}{q} = \log_2 x + O(1)$$ and thus $$ \sum_{\substack{ q\in \mathcal{P} \\ x^{1/v}<q\leq x^{1/u}}} \frac 1q \geq \frac{1}{2} \log(v/u) .$$ Consequently, using (\ref{decompprimes}) we get that there exists $i \in \{1,\dots,k\}$ such that

\begin{equation}\label{inverseprimes} \sum_{\substack{ q\in \mathcal{P}_{i} \\ x^{1/v}<q\leq x^{1/u}}} \frac 1q \geq \frac{1}{2k} \log(v/u) .\end{equation} We want to apply Theorem \ref{th:MT}, hence we need the right hand side of (\ref{inverseprimes}) to be larger than $\frac{1+\epsilon}{u}$ under the conditions $1 \leq u \leq v\le \frac{\log x}{1000 \log_2  x}$. Let us fix $u$ such that $\frac{1}{u}=\frac{1}{4\sqrt{e}}+\epsilon$ and $x=p$ so that $N=x^{1/u}$.
 Thus the condition of Theorem \ref{th:MT} is verified as long $k\leq \frac{\log v}{4}$. Therefore, we get 

$$\frac{\Psi(p;\mathcal{P}_{i})}{p} \geq A_v \prod_{q\in \mathcal{P}_{i}^c}  \left( 1 -\frac 1q \right).$$ Using the third Mertens' Theorem, the product is trivially bounded from below by $$\displaystyle \prod_{q \leq p}  \left( 1 -\frac 1q \right)\geq \frac{1}{\log p}$$ for $p$ large enough. Thus, we obtain the inequality $\Psi(p;\mathcal{P}_{i}) \gg A_v(\log p)^{-1}x \gg v^{-v}\frac{x}{\log p}$. On the other hand, we are counting integers less than $p$ which are $q_1^{th}$ residues and so there are at most $p/q_i$ of these. It leads to a contradiction when $v^{-v}(\log p)^{-1} \geq 1/q_i$    or equivalently $q_i\geq (\log p)v^v$. In this case, the set $S$ is non empty and this concludes the proof setting $v=l$.

\end{proof}

This proposition helps us to regroup the divisors in ``blocks" of size $k$.  Using this idea in a simple way, we are able to deduce the result announced in the introduction:

\begin{theorem}\label{general} Let $l:=l(p)\geq 1$ a parameter tending to infinity with $p$ such that $l\leq \frac{\log p}{1000 \log_2 p}$. For $p$ a sufficiently large prime, we have the bound  $$h(p) \ll \omega_l(p-1) + \frac{\omega(p-1)-\omega_l(p-1)}{\log l}$$ where the implied constant is effective. 
\end{theorem}

\begin{proof} Consider the prime divisors of $p-1$ which are greater than $(\log p)l^l$. We can apply the Proposition \ref{blocsk} with $k=\frac{\log p}{4}$ and pick up an integer less than $p^{1/4\sqrt{e}+\epsilon}$ which is a non $q^{th}$ residue for $k$ different large $q$. Regrouping the large divisors of $p-1$ in blocks of size $k$, we have at most $\frac{\omega(p-1)-\omega_l(p-1)}{k}$ of such blocks. This concludes the proof including the contribution of small divisors treated individually using Burgess' character sums inequality combined with ``Vinogradov trick" as in Lemma \ref{elementary}. \end{proof}

\begin{remark} The value of the optimal parameter $l$ is not so clear for a general $p$, it will depends heavily on the repartition of the prime divisors of $p-1$. 
\end{remark} We can in fact iterate in some sense the argument used to prove Theorem \ref{general} and obtain the following stronger result:


\begin{theorem}\label{iteration}
Let $l_n(p), n=0,\dots,N$ be a strictly decreasing sequence of parameters tending to infinity with $p$ such that $(\log p)l_0^{l_0}>p$ and that $l_1\leq \frac{\log p}{1000 \log_2 p}$. Then, for $p$ a sufficiently large prime, we have

$$h(p) \ll  \omega_{l_{N}}(p-1)+\sum_{n=0}^{N-1} \frac{\omega_{l_n}(p-1)-\omega_{l_{n+1}}(p-1)}{\log(l_{n+1})}.      $$

\end{theorem}

\begin{proof} We argue similarly as in Theorem \ref{general}, regrouping the divisors of $p-1$ lying in the interval $](\log p)l_{n+1}^{l_{n+1}},(\log p)l_n^{l_n}]$ in blocks of size $k_n \approx \log (l_{n+1})$. The contribution of the remaining small prime divisors is given by $\omega_{l_{N}}(p-1)$.
 \end{proof}

Even though stronger results about primitive roots are known for almost all primes, a result on a set of primes of density $1$ follows as a consequence of Theorem \ref{iteration}.

\subsection{Results for almost all primes}

The next result gives a bound on the number of small prime divisors of $p-1$ for most of the primes $p$.

\begin{lemma}\label{omega} Let $A>1$ and $\epsilon>0$. Suppose $l$ is such that $l^l \ll x^{1/2-\epsilon}$. Then, the set of primes $p\leq x$ such that $p-1$ verifies $\omega_l(p-1) \ll \log l$ is asymptotically of density $1$. \end{lemma}

\begin{proof} We evaluate the average number of primes verifying the inequality of the lemma: \begin{equation*}\sum_{p\leq x \atop p \text{ prime }} \sum_{q\vert p-1 \atop q\leq (\log p)l^l, q \text{ prime }} 1= \sum_{q \leq (\log x)l^l} \sum_{p\equiv 1 \bmod q \atop p\leq x} 1 = \sum_{q\leq (\log x)l^l} \frac{x}{(q-1)\log x} + O\left(\frac{x}{\log^A x}\right)  \end{equation*} where we used the Bombieri-Vinogradov theorem (see for instance \cite{Bombieri}). Thus, using Mertens' Theorem, it gives

$$\sum_{p\leq x \atop p \text{ prime }} \sum_{q\vert p-1 \atop q\leq (\log p)l^l, q \text{ prime }} 1=\frac{x}{\log x} (\log l + \log_2 l + M) + O\left( \frac{x}{\log^B x}\right)$$ where $M$ is the Meissel-Mertens constant and $B=\min\left\{A,2\right\}$. The conclusion follows easily.

\end{proof}

\begin{remark} We could obtain the normal order of $\omega_l(p-1)$ following the method of Turan (see \cite{Turan}) using the first two moments. It might be even possible to prove a more precise statement like an Erd\"{o}s-Kac version of this result using the method of Granville and Soundararajan (see \cite{Kac}) but it is not the purpose of this note. \end{remark}

\begin{corollary}\label{dyadicappli} For almost all primes $p$, we have $h(p)\ll (\log_3 p)^2$.\end{corollary}

\begin{proof} 

In order to prove this result, we define dyadically some special parameters. Let $l_n=\exp\left(\frac{\log_2 p}{2^n\log_3 p}\right)$ for $1\leq n\leq N=\frac{\log_3 p-2\log_4 p}{\log 2}$. It is easy to see that this sequence fullfils the hypotheses of Theorem \ref{iteration}, thus we derive

$$ h(p) \ll  \omega_{l_{N}}(p-1)+\sum_{n=1}^{N-1} \frac{\omega_{l_n}(p-1)-\omega_{l_{n+1}}(p-1)}{\log(l_{n+1})}+\frac{\omega(p-1)-\omega_{l_1}(p-1)}{\log (l_1)}.$$ We remark by using Lemma \ref{omega} that the bound $\omega_{l_n}(p)\leq \log (l_{n})(\log_3 p)$  holds for almost all primes $p\leq x$ with an exceptional set of  ``bad" primes of size at most $\frac{x}{\log x \log_3 x}$. Applying $N$ times Lemma \ref{omega}, we end up with a set of primes of density $1$ verifying $\omega_{l_n}(p-1)\leq \log (l_{n})(\log_3 p)$ for all $1\leq n\leq N$ with a negligeable exceptional set of ``bad" primes.  
Using the trivial inequality $\frac{\log (l_{n})}{\log (l_{n+1})} \leq 2$  this leads to

$$ h(p) \leq \log (l_{N})\log_3 p + 2N \log_3 p+ \log_3 p \ll (\log_3 p)^2$$ on a set of primes of density $1$.

\end{proof}

\begin{remark} As an application of large sieve, Pappalardi obtained a similar flavour type of result. Precisely, in \cite{Pappagenerating}, he showed that the first $\frac{\log^2 p}{\log_2 p}$ primes generate a primitive root modulo $p$ for almost all primes $p$. \end{remark}

 \section*{Acknowledgements}

The author is grateful to Sary Drappeau and Igor E. Shparlinski for very helpful discussions and is supported by the Austrian Science Fund (FWF), START-project Y-901 ``Probabilistic methods in analysis and number theory" headed by Christoph Aistleitner.

\bibliographystyle{plain}

\end{document}